\documentclass{amsart}
\pdfoutput=1

\usepackage[T1]{fontenc}
\usepackage[utf8]{inputenc}
\usepackage[english]{babel}
\usepackage{amsfonts,amsmath,amssymb, amsthm}
\usepackage{mathtools}
\usepackage{tikz}
\usepackage{xspace}
\usepackage[colorlinks,linkcolor=links,citecolor=cites]{hyperref}
\usepackage{aliascnt}
\usepackage{varioref}
\usepackage{lmodern}
\usepackage{microtype}
\usepackage[shortlabels]{enumitem}

\microtypesetup{tracking,kerning,spacing}
\microtypecontext{spacing=nonfrench}

\usetikzlibrary{calc}

\definecolor{links}{rgb}{.2,.1,.5}
\definecolor{cites}{rgb}{.5,.1,.2}

\newcommand\QQ{{\mathbb Q}}
\newcommand\RR{{\mathbb R}}
\newcommand\ZZ{{\mathbb Z}}
\newcommand\PP{{\mathbb P}}

\newcommand{\pro}[2]{\langle #1, #2 \rangle} % scalar product
\newcommand{\SetOf}[2]{\left\{#1\vphantom{#2}\,\right.\left|\,\vphantom{#1}#2\right\}}

\renewcommand\Vert{\operatorname{Vert}} % vertices of a polytope

\DeclareMathOperator\conv{conv}
\DeclareMathOperator\opp{opp}
\DeclareMathOperator\Hom{Hom}

\newcommand\polymake{\texttt{polymake}\xspace}
\newcommand\Obro{{\O}bro\xspace}
\newcommand\neigh{\operatorname{N}} % neighboring facet
\AtBeginDocument{
  \setlength{\leftmargini}{2em}
  \setlength{\leftmarginii}{2em}
  \setlength{\leftmarginiii}{2em}
  \setlength{\leftmarginiv}{2em}
}

\newcommand{\theoremname}{Theorem}
\newcommand{\corollaryname}{Corollary}
\newcommand{\lemmaname}{Lemma}
\newcommand{\propositionname}{Proposition}
\newcommand{\conjecturename}{Conjecture}
\newcommand{\remarkname}{Remark}
\newcommand{\examplename}{Example}
\newcommand{\definitionname}{Definition}
\newcommand{\questionname}{Question}
\newcommand{\answername}{Answer}

\theoremstyle{plain}
    \newtheorem{theorem}{\theoremname}
    \newtheorem*{theorem*}{\theoremname}
  \newaliascnt{corollary}{theorem}
    \newtheorem{corollary}[corollary]{\corollaryname}
  \newaliascnt{lemma}{theorem}
    \newtheorem{lemma}[lemma]{\lemmaname}
  \newaliascnt{proposition}{theorem}
    \newtheorem{proposition}[proposition]{\propositionname}
\theoremstyle{definition}
  \newaliascnt{conjecture}{theorem}
    \newtheorem{conjecture}[conjecture]{\conjecturename}
  \newaliascnt{remark}{theorem}
    \newtheorem{remark}[remark]{\remarkname}
    \newtheorem*{remark*}{\remarkname}
  \newaliascnt{example}{theorem}
    \newtheorem{example}[example]{\examplename}
  \newaliascnt{definition}{theorem}
    \newtheorem{definition}[definition]{\definitionname}
  \newaliascnt{question}{theorem}
    
  \newaliascnt{answer}{theorem}
    
    \newtheorem*{acknowledgment}{Acknowledgments}

\aliascntresetthe{corollary}
\aliascntresetthe{lemma}
\aliascntresetthe{proposition}
\aliascntresetthe{conjecture}
\aliascntresetthe{remark}
\aliascntresetthe{example}
\aliascntresetthe{definition}
\aliascntresetthe{question}
\aliascntresetthe{answer}

\title[A bound for the splitting of smooth Fano polytopes with many vertices]{A bound for the splitting of smooth Fano polytopes with many vertices}

\author{Benjamin Assarf}
\address{
Benjamin Assarf \newline \mbox{ }\quad
TU-Berlin, Str. des 17. Juni 136, D-10623 Berlin, Germany
}
\email{assarf@math.tu-berlin.de}

\author{Benjamin Nill}
\address{
Benjamin Nill \newline \mbox{ }\quad
Stockholm University, Kräftriket,  SE-10691 Stockholm, Sweden
}
\email{nill@math.su.se}


\begin{document}

\begin{abstract}
The classification of toric Fano manifolds with large Picard number corresponds to the classification of smooth Fano polytopes with large number of vertices. A smooth Fano polytope is a polytope that contains the origin in its interior such that the vertex set of each facet forms a lattice basis. Casagrande showed that any smooth $d$-dimensional Fano polytope has at most $3d$ vertices. Smooth Fano polytopes in dimension $d$ with at least $3d-2$ vertices are completely known. The main result of this paper deals with the case of $3d-k$ vertices for $k$ fixed and $d$ large. It implies that there is only a finite number of isomorphism classes of toric Fano $d$-folds $X$ (for arbitrary $d$) with Picard number $2d-k$ such that $X$ is not a product of a lower-dimensional toric Fano manifold and the
projective plane blown up in three torus-invariant points. This verifies the qualitative part of a conjecture in a recent paper by the first author, Joswig, and Paffenholz.
\end{abstract}

\maketitle
\section{Introduction and main results}

Let us first recall the basic definitions. We refer to \cite{Fanosurvey,Ewald} for more background. Let $N\cong\ZZ^d$ be a lattice with associated real vector space $N_\RR:=N\otimes_\ZZ\RR$ isomorphic to $\RR^d$. A \emph{polytope} $P$ is a convex, compact set in $N_\RR$, its $0$-dimensional faces are called \emph{vertices}, and its faces of codimension $1$ are called \emph{facets}. If every facet $F$ (of dimension $d-1$) of a $d$-dimensional polytope $P$ has exactly $d$ vertices (i.e., $F$ is a simplex), then $P$ is called \emph{simplicial}. The polytope $P$ is called a \emph{lattice polytope} if its vertices are lattice points
(i.e., elements of $N$).

\begin{definition}
  \begin{enumerate}
    \item A polytope $P\subset N_\RR$ is called \emph{smooth Fano polytope}, if
      \begin{itemize}
        \item $P$ is a lattice polytope, and
        \item $P$ is full-dimensional and contains the origin $0$ as an interior point, and
        \item for each facet $F$ of $P$, the vertex set $\Vert{F}$ is a lattice basis of $N$.
      \end{itemize}

    \item Two smooth Fano polytopes are \emph{lattice equivalent}, if their vertex sets are in bijection by an affine-linear lattice automorphism.
\end{enumerate}
\end{definition}

\begin{remark}
We decided to keep the notion of a smooth Fano polytope in order to be consistent with existing literature. However, we remark that there exists also the definition of a \emph{smooth polytope} as a lattice polytope with unimodular vertex cones. A smooth Fano polytope is \emph{not} a smooth polytope (but its dual polytope is).
\end{remark}

Note that any smooth Fano polytope $P$ is necessarily simplicial. In each dimension there exist only finitely many smooth Fano polytopes up to lattice equivalence (we refer to the survey
\cite{Fanosurvey}). In 2007, \Obro described an explicit classification algorithm in each dimension $d$, see \cite{Obro-alg,OebroPhD}. His implementation gave complete classifications of lattice
equivalence classes of smooth Fano polytopes up to dimension $8$ extending the previous existing classification results up to dimension $5$; see \cite{Batyrev81,WW82,Batyrev99,Sato00,NillKreuzer09}. Using a variant of \Obro's algorithm, the classification in dimension~$9$ was done by Lorenz and Paffenholz~\cite{smoothreflexive}. The large number of $8\,229\,721$ smooth Fano polytopes in dimension $9$ indicates that complete classifications in much higher dimensions are not feasible. This motivates to focus on finding sharp bounds on important invariants and to classify the extreme cases.

\smallskip

When considering smooth Fano polytopes with large number of vertices, one needs to start with dimension two.
Let $P_6$ be the smooth Fano polygon with $6$ vertices as given in \autoref{hexagon-fig}. This polygon has the maximal number of vertices possible in dimension $2$.

\begin{figure}[tb]
  \centering
  \begin{tikzpicture}[scale=0.9]
    \tikzstyle{edge} = [draw,thick,-,black]

    \foreach \x in {-1,0,1}
    \foreach \y in {-1,0,1}
    \fill[gray] (\x,\y) circle (1.5pt); ;

    \coordinate (v0) at (0,0);
    \coordinate (e1) at (1,0);
    \coordinate (e2) at (0,1);
    \coordinate (-e1+e2) at (-1,1);
    \coordinate (-e2+e1) at (1,-1);
    \coordinate (-e1) at (-1,0);
    \coordinate (-e2) at (0,-1);

    \draw[edge] (e1) -- (e2) -- (-e1+e2) -- (-e1) -- (-e2) -- (-e2+e1) -- (e1);

    \foreach \point in {e1,e2,-e1+e2,-e2+e1,-e1,-e2}
    \fill[black] (\point) circle (2pt);

  \end{tikzpicture}
  \caption{\label{hexagon-fig} The smooth Fano polygon $P_6$ corresponding to the del Pezzo surface $S_3$}
\end{figure}
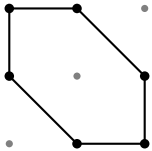

There is a natural direct sum operation on smooth Fano polytopes (also sometimes called free sum).

\begin{definition} Given smooth Fano polytopes $Q \subset N_\RR$ and $Q' \subset N'_\RR$, we define their \emph{direct sum} as
  \[
    Q\oplus Q'\ = \ \conv(Q \times \{0\} \cup \{0\} \times Q') \quad \subset \ N_\RR \times N'_\RR.
  \]
  Direct sums of smooth Fano polytopes are again smooth Fano polytopes.
  A smooth Fano polytope $P$ \emph{splits} into $Q$ and $Q'$ if $P$ is lattice equivalent to $Q \oplus Q'$.
  \label{free-sum-def}
\end{definition}

In 2006 Casagrande proved the following long-standing conjecture.

\begin{theorem}[Casagrande~\cite{Casagrande06}]\label{casa:theorem}
  Let $P$ be a $d$-dimensional smooth Fano polytope. Then $P$ has at most $3d$ vertices, equality holds if and only if $d$ is even and $P$ is lattice equivalent to $(P_6)^{\oplus d/2}$.
\end{theorem}

By now, smooth Fano polytopes with at least $3d-2$ vertices are completely classified \cite{Obro08,AJP}. In \cite[Conjecture 9]{AJP} it was conjectured that in order to classify smooth
Fano polytopes $P$ of dimension $d$ with $|\Vert{P}| = 3d-k$ it is enough to do this up to dimension $3k$. More precisely, if $d > 3k$, then $P$ should be a direct sum of
$P_6$ and a lower-dimensional smooth Fano polytope. The goal of this paper is to verify the qualitative part of this conjecture.

\begin{theorem}\label{thm:main}
  Let $P$ be a $d$-dimensional smooth Fano polytope with $3d-k$ vertices and $d \ge 15k^2+37k + 2$. Further let  $f(d,k) := \lfloor\frac{d-15k^2-37k}{2}\rfloor$. Then $P$ is lattice equivalent to
    \[
      Q \oplus P_6^{\oplus f(d,k)}
    \]
  where $Q$ is a smooth Fano polytope of dimension $d-2 f(d,k)$.
\end{theorem}

\begin{remark}  Any smooth Fano polytope is a \emph{reflexive polytope}, i.e., it contains the origin in its interior (as the only interior lattice point) and its dual polytope
is also a lattice polytope (we refer to \cite{Bat94,1067.14052}). Reflexive polytopes always appear as dual pairs. They correspond to Gorenstein toric Fano varieties and were introduced by Batyrev in 1994 to provide a combinatorial framework for constructing mirror symmetric pairs of Calabi-Yau hypersurfaces \cite{Bat94}. In fixed dimension only a finite number of reflexive polytopes exist up to lattice equivalence. Still, Haase and Melnikov showed that {\em any} lattice polytope is isomorphic to the face of some (possibly much higher-dimensional) reflexive polytope, \cite{HM06}. Interestingly, it is conjectured that $P_6^{d/2}$ has the maximal number of vertices of a reflexive polytope in even dimension; check \cite{B++}. Note that this is the dual reflexive polytope of $P_6^{\oplus d/2}$.
\end{remark}

\subsection{The algebro-geometric viewpoint: background, applications, and conjectures}
\label{sec:TORIC}
To make this paper accessible not only to combinatorialists but also to algebraic geometers, we decided to split the introduction. In this section we will use algebro-geometric language
to describe the relevance in toric geometry, to discuss the equality case, and to present more detailed conjectures (all toric geometry statements can easily be translated into combinatorial ones). However, note that all proofs in the following sections will be purely combinatorial.
\smallskip

While smooth Fano polytopes are interesting peculiar classes of lattice polytopes, the main reason for their investigation originates in algebraic geometry. Toric Fano manifolds are among the most intensively studied classes of toric varieties. We refer to the surveys \cite{Debarre03} and \cite{Fanosurvey} and the references therein. By the toric dictionary (e.g. \cite[Chapter 8]{CLSbook}), any $d$-dimensional smooth Fano polytope $P$ corresponds one-to-one to a $d$-dimensional \emph{toric Fano manifold} $X$. Two smooth Fano polytopes are lattice equivalent if and only if the corresponding toric Fano manifolds are isomorphic.
In each dimension $d$, there exists only a finite number of toric Fano $d$-folds up to isomorphisms. As described in the previous section, they are completely classified up to dimension $9$. Their large number in
higher dimensions motivates the study of subclasses of special interest.

Let $X$ be a toric Fano $d$-fold corresponding to a smooth Fano polytope $P$ of dimension $d$. Then the \emph{Picard number} $\rho_X$ of $X$ equals $|\Vert{P}|-d$.
For instance, $P_6$ corresponds to the del Pezzo surface $S_3$ ($\PP^2$ blown up in three torus-invariant points), note that $\rho_{S_3} = 4$.
A smooth Fano polytope $P$ splits into two smooth Fano polytopes if and only if $X$ is isomorphic to the product of the corresponding toric Fano manifolds.

\autoref{casa:theorem}, by Casagrande, implies that the Picard number $\rho_X$ of any toric Fano $d$-fold $X$ is bounded from above by $2d$.
Moreover, the equality case is only attained in even dimension $d$ by $(S_3)^{d/2}$, the product of $d/2$ copies of the del Pezzo surface $S_3$. In 2008, \Obro showed that if the Picard number equals $2d-1$, then the variety $X$ is isomorphic to $(S_3)^{(d-2)/2} \times S_2$  if $d$ is even, where $S_2$ is $\PP^2$ blown up in two torus-invariant points, and $X$ is
isomorphic to $(S_3)^{(d-3)/2} \times Y$ if $d$ is odd, where $Y$ is one of two toric Fano $3$-folds with $\rho_Y = 5$. Recently, a complete classification of toric Fano $d$-folds $X$ with $\rho_X = 2d-2$ was given by Assarf, Joswig, Paffenholz. Their result gave rise to the following conjecture:

\begin{conjecture}[{\cite[Conjecture 9]{AJP}}]\label{conjecture}
  Let $X$ be a toric Fano $d$-fold with $\rho_X = 2d-k$. If $d > 3k$, then $X$ is isomorphic to $(S_3)^m \times Y$,
  where $m$ is a positive integer and $Y$ is a toric Fano manifold of dimension at most $3k$.
\end{conjecture}

Note that necessarily here $\rho_Y = 2 \dim(Y) - k$, since $\rho_{S_3}=4$. As will be explained below, this conjecture would be best possible. It would generalize the mentioned classification results, moreover, its validity would reduce the classification of toric Fano $d$-folds with $\rho_X = 2d-3$ to the case of $d \le 9$, where complete classifications already exist.

The main result of this paper is the verification of the following weaker version of \autoref{conjecture}.

\begin{theorem}\label{main-alg-geo}
  Let $X$ be a toric Fano $d$-fold with $\rho_X = 2d-k$. If $d > 15k^2+37k + 1$, then $X$ is isomorphic to $(S_3)^m \times Y$,
  where $m$ is a positive integer and $Y$ is a toric Fano manifold of dimension at most $15k^2+37k + 1$.
\end{theorem}

This result is equivalent to \autoref{thm:main}.
Let us remark that it seems to be a hard open problem to improve the quadratic bound in \autoref{main-alg-geo} to a linear bound (as expected by \autoref{conjecture}). See also \autoref{remark:bottle-neck}.

\begin{remark}
  Casagrande's upper bound of $2d$ on the Picard number was proven in more generality for $\QQ$-factorial Gorenstein toric Fano varieties \cite{Casagrande06} (corresponding to simplicial reflexive polytopes). In \cite{Nill-Obro} \Obro and the second author completely classified all
  such varieties with $\rho_X =2d-1$. These results suggest that \autoref{conjecture} (and \autoref{main-alg-geo}) should also hold when $X$ and $Y$ are $\QQ$-factorial Gorenstein toric Fano varieties.
\end{remark}

We say a toric Fano manifold $X$ is \emph{non-$S_3$-splittable}, if it is not a product of $S_3$ and a lower-dimensional toric Fano manifold.
Since there are only finitely many toric Fano $d$-folds of dimension $d \le 15k^2+37k + 1$, we get the following consequence from \autoref{main-alg-geo}.

\begin{corollary}
  For each $k$ there is only a finite number of isomorphism classes of non-$S_3$-splittable toric Fano $d$-folds $X$ with $\rho_X \ge 2d-k$.
\end{corollary}

Let us illustrate the previous result by considering the number of isomorphism classes of non-$S_3$-splittable toric Fano manifolds for the known cases $k=0,1,2$. For $k=0$ we have just $1$ such manifold (namely a point),
for $k=1$ there are $3$ (one in each dimension $1$, $2$, and $3$), and for $k=2$ there are $15$ isomorphism classes where the highest dimension is $6$ (the numbers are $2,4,7,1,1$ from dimension $2$ up to dimension $6$ respectively).

\smallskip

Regarding the sharpness of \autoref{conjecture}, it is natural to reformulate and extend it in the following way (for an explicit description of the extreme case see \hyperref[sec:extreme]{Section~\ref*{sec:extreme})}.

\begin{conjecture}\label{conjecture1'}
  Let $X$ be a non-$S_3$-splittable toric Fano $d$-fold. Then $\rho_X \le \frac{5d}{3}$, with equality if and only if $d = 3 d'$ (for $d'$ a positive integer) and
  $X$ is isomorphic to the $d'$-fold product of a uniquely determined toric $S_3$-bundle over $\PP^1$.
\end{conjecture}

\begin{remark}\label{alg-geo-remark}
  In the non-toric setting much less is known about the structure of Fano manifolds with maximal Picard number. It is known that Fano $3$-folds have Picard number at most 10 with equality only for the product of a Del Pezzo surface with $\PP^1$. One conjecture in higher dimensions is that in even dimension $d$ the maximal Picard number equals $\frac{9d}{2}$ and is attained only for products of del Pezzo surfaces with Picard number $9$. There are further results \cite{prime-divisor,dim-four} that indicate that also in the non-toric setting extremizing the Picard number naturally leads to fibrations of Del Pezzo surfaces over another Fano manifold.
\end{remark}

Let us finish this introduction by proposing the following further refinement of Casagrande's upper bound \cite{Casagrande06}, based upon classification results in low dimensions. We say a toric Fano manifold $X$ is \emph{non-splittable} if it is not a product of
two lower-dimensional toric Fano manifolds.

\begin{conjecture}\label{conjecture2}
  Let $X$ be a non-splittable toric Fano $d$-fold, where $d \ge 3$. Then $\rho_X \le \frac{4d+3}{3}$, with equality if and only if $d = 3 d'$ (for $d'$ a positive integer) and
  $X$ is isomorphic to a uniquely determined toric $(S_3)^{2d'}$-bundle over $\PP^{d'}$.
\end{conjecture}

As mentioned above smooth Fano polytopes are already classified up to dimension~$9$. The conjectures are true for $d\le 9$; see \hyperref[sec:extreme]{Section~\ref*{sec:extreme}}.

\subsection{Organization of the paper}

This paper is organized as follows. \hyperref[sec:last]{Section~\ref*{sec:last}} contains the combinatorial proof of \autoref{main-alg-geo}. \hyperref[sec:extreme]{Section~\ref*{sec:extreme}} discusses the equality cases in \autoref{conjecture1'} and \autoref{conjecture2}.

\section{Proof of the main theorem}
\label{sec:last}

\subsection{Preliminaries of the proof}

Let $P$ be a $d$-dimensional smooth Fano polytope and $F$ a facet of $P$. Let us define some important notions and fix the notation (we refer to \cite{OebroPhD,1067.14052}).

For every vertex $v \in F$ there exists a unique \emph{neighboring facet} $\neigh(F,v)$ that intersects with $F$ in a $(d-2)$-dimensional face that does not contain $v$. The unique vertex of $\neigh(F,v)$ that is not contained in $F$ is called the \emph{opposite vertex} of $v$ with respect to $F$ and will be denoted by $\opp(F,v)$; see \autoref{fig:opposite_vertex}.

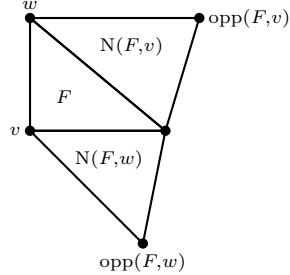
\begin{figure}[ht]
  \centering
  \begin{tikzpicture}[scale=1.5]
    \tikzstyle{edge} = [draw,thick,-,black]

    \coordinate[label=left:$\scriptstyle v$] (v) at (0,0);
    \coordinate (z) at (1.2,0);
    \coordinate[label=above:$\scriptstyle w$] (w) at (0,1);
    \coordinate[label={right:$\scriptstyle \opp(F,v)$}] (ov) at (1.5,1);
    \coordinate[label={below:$\scriptstyle \opp(F,w)$}] (ow) at (1,-1);

    \draw[edge] (v) -- (w) -- (z) -- (v) -- cycle;
    \draw[edge] (ov) -- (w) -- (z) -- (ov) -- cycle;
    \draw[edge] (v) -- (ow) -- (z) -- (v) --cycle;

    \foreach \point in {v,z,w,ov,ow}
      \fill[black] (\point) circle (1.3pt);

    \draw (.3,.3) node {$\scriptstyle F$};
    \draw (.9,.75) node {$\scriptstyle \neigh(F,v)$};
    \draw (.7,-.25) node {$\scriptstyle \neigh(F,w)$};
  \end{tikzpicture}
  \caption{\label{fig:opposite_vertex} Neighboring facets and opposite vertices}
\end{figure}

We define the \emph{dual lattice} of $N$ as $M:=\Hom_\ZZ(N,\ZZ)\cong \ZZ^d$ with the pairing $\pro{\cdot}{\cdot} \,:\, M \times N \to \ZZ$. For $\Vert{F} = \{v_1, \ldots, v_d\}$, a lattice basis of $N$, we get the dual lattice basis
$u_{F,v_1}, \ldots, u_{F,v_d}$ of $M$. In other words, for $v \in \Vert{F}$ the lattice point $u_{F,v} \in M$ is characterized by $\pro{u_{F,v}}{v} = 1$ and $\pro{u_{F,v}}{w}=0$ for $w \in \Vert{F}\setminus\{v\}$. Moreover, it holds that $u_F := \sum_{i=1}^d u_{F,v_i} \in M$ is the unique primitive \emph{outer normal} of $F$, i.e., for every $x \in \Vert{F}$ we have $\pro{u_F}{x} = 1$ and for every $y \in P$ we have $\pro{u_F}{y} \le 1$.

Let $j \in \ZZ$. We define the set of vertices of $P$ of \emph{level} $j$ as
  \[
    V(F,j) := \{v \in \Vert{P} \;:\; \pro{u_F}{v} = j\}\,.
  \]
Moreover, we set $\eta^F_j := |V(F,j)|$.
Note that $\eta^F_1 = d$ as $P$ is smooth and therefore simplicial.
The vector $\eta^F := (\eta^F_1, \eta^F_0, \eta^F_{-1}, \ldots)$ is called the \emph{$\eta^F$-vector} of $P$.

An important role is played by the \emph{vertex sum} $s_P$ of $P$:
  \[
    s_P := \sum_{v\in \Vert P} v\,.
  \]
We say that $F$ is a \emph{special facet} \cite{Obro08}, if $s_P$ is in the cone spanned by $F$. Note that in this case necessarily $\pro{u_F}{s_P} \ge 0$. Special facets were the crucial
tool in \Obro's classification algorithm for smooth Fano polytopes \cite{Obro-alg}.

\subsection{General assumptions of the proof}
By a suitable lattice equivalence we can assume that we are from now on in the following situation:

\begin{itemize}
  \item $P$ is a $d$-dimensional smooth Fano polytope,

  \item $|\Vert P| = 3d-k$ where $d \ge 3$ and $k \ge 3$ (for the case $k\le 2$ see \cite{Casagrande06,Obro-alg,AJP}),

  \item $F$ denotes a special facet which has outer normal vector $u_F$,

  \item we write $\eta_j := \eta^F_j$ for the number of vertices of $F$ on level $j$. We also use the notation $\eta_{\le j} := \sum_{i \le j} \eta_i$.
\end{itemize}

\subsection{Distribution of the vertices}
Let us start by recalling some basic properties of vertices on level $0$:

\begin{lemma}[{Nill~\cite[Lem.~5.5]{1067.14052}}]\label{lem:nill_5.5}
  Let $x \in V(F,0)$. Then there exists a vertex $v \in \Vert F$ such that $x = \opp(F,v)$. Moreover, for $v \in \Vert F$ we have
  \[
    x = \opp(F,v) \quad \iff \quad \langle u_{F,v},x \rangle < 0 \quad \iff \quad \langle u_{F,v},x \rangle = -1.
  \]
\end{lemma}

Here the second equivalence follows from the smoothness of $P$ (for a more general statement see \autoref{minus}, where the opposite vertex is not necessarily assumed
to be on level zero).

\begin{lemma}\label{card}
  The $\eta$-vector of any special facet of $P$ has the following properties:
    \begin{align*}
      \eta_1 &= d, \\
      d-k \le \eta_0 &\le d, \\
      d-2k \le \eta_{-1} &\le d, \\
      \eta_{\le -2} &\le 2k, \\
      \eta_j &= 0 \text{ for } j < -k-1.
    \end{align*}
\end{lemma}

\begin{proof}
  The first equation is trivial, because $P$ is a smooth Fano polytope and therefore simplicial.
  The upper bound on $\eta_0$ follows from \autoref{lem:nill_5.5}, since there are at most $d$ opposite vertices to the $d$ vertices in a given facet. The upper bound for $\eta_{-1}$ comes from the fact that we are looking at a special facet $F$ and therefore we have
    \[
      0 \le \langle u_F, s_P \rangle = \sum_{i \le 1} i \cdot \eta_i = d + \sum_{i \le -1}i \cdot \eta_i \,.
    \]
  Now the upper bound follows from:
    \[
      d \ge \sum_{i \le -1} |i|\cdot \eta_i \ge \sum_{i\le -1} \eta_{i} \,.
    \]
  With this inequality we also get the lower bound on $\eta_0$ since $|\Vert P| = \sum_{i\le 1} \eta_i$ gives us
    \[
      3d-k = d + \eta_0 + \sum_{i\le -1}\eta_i \le d + \eta_0 + d\,.
    \]
  Solving this for $\eta_0$ we get $d-k \le \eta_0$.

  The lower bound for $\eta_{-1}$ follows with a similar argument. From the equality $|\Vert P| = \sum_{i\le 1} \eta_i$ we get
    \begin{align*}
      3d-k - \eta_1 - \eta_0 - \eta_{-1} &= \sum_{i\le -2}\eta_i \,.
    \end{align*}
  Using the fact that we have a special facet and using the upper bounds on $\eta_0$ and $\eta_1$ gives us
    \begin{align*}
      0 &\le \langle u_F, s_P \rangle \\
        &= \sum_{i \le 1} i \cdot \eta_i \\
        &= d - \eta_{-1} + \sum_{i \le -2} i \cdot \eta_i \\
        &\le d - \eta_{-1} -2\sum_{i\le -2} \eta_i \\
        &= d - \eta_{-1} -2 (3d-k - \eta_1 - \eta_0 - \eta_{-1}) \\
        &\le d - \eta_{-1} -6d +2k + 2d +2d +2\eta_{-1} \\
        &= \eta_{-1} - d + 2k \,.
    \end{align*}
  Now, the upper bound for $\eta_{\le -2}$ is a direct consequence of the previous lower bounds:
    \begin{align*}
      \sum_{i\le -2}\eta_i &= 3d-k - \eta_1 - \eta_0 - \eta_{-1}\\
      & \le 3d-k -d -(d-k) -(d-2k) =2k \,.
    \end{align*}
  Finally, let $j < 0$ be the minimal level with $\eta_j \not= 0$. Let $v \in V(F,j)$. Then there are at least $3d-k-\eta_1 - \eta_0-1 = 2d-k-1-\eta_0$ vertices on negative levels
  apart from $v$. In particular, 
   \begin{align*}
    0 &\le \pro{u_F}{s_P}\\
      &\le d-(2d-k-1-\eta_0)+j \\
      &= -d+k+1+\eta_0+j\,,
   \end{align*}
  hence $j \ge d-\eta_0-k-1 \ge -k-1$, since $\eta_0 \le d$.
\end{proof}

\begin{lemma}\label{sum-level}
  The level of $s_P$ is at most $k$.
\end{lemma}

\begin{proof}
  The maximal level of $s_P$ is achieved by having $d$ vertices on level $1$, $d$ vertices on level $0$, and the remaining vertices ($d-k$ many) on level $-1$.
\end{proof}

\subsection{Partitioning the vertex set of \texorpdfstring{$F$}{F}}
We will use the following notation introduced in \cite{AJP}.

\begin{definition}
  We call a vertex $v$ of the facet $F$ \emph{good} if $\opp(F,v)$ is contained in $V(F,0)$.
\end{definition}

As the following well-known result shows, the second condition of a good vertex given in \cite{AJP} is automatically true in our setting because we assume smoothness.

\begin{lemma}\label{minus}
  If $v$ is a vertex of $F$, then $\langle u_{F,v},\opp(F,v) \rangle=-1$.
\end{lemma}

%\begin{proof}
%  By definition, $\opp(F,v) = \sum_{w \in \Vert F} \langle u_{F,w},\opp(F,v) \rangle \,w$. Clearly, since $0$ is in the interior of $P$, $\langle u_{F,v},\opp(F,v) \rangle < 0$. Because $\opp(F,v)$ together with the vertices $w \not= v$ of $\Vert F$ is the vertex set of a facet of $P$ and hence a lattice basis, we deduce $\langle u_{F,v},\opp(F,v) \rangle =-1$.
%\end{proof}

Let us also recall the notation of the $\phi$-function (of a facet $F$) used in \cite{AJP}:

\begin{equation}\label{eq:phi}
  \phi \,:\, \Vert{F}\to\Vert{F}\cup\{0\} \,,\; v\mapsto\begin{cases} w & \text{if }
    w:=\opp(F,v)+v\in\Vert{F} \\ 0 & \text{otherwise\,.}\end{cases}
\end{equation}

Note that if $\phi(v) \not=0$, then $v$ is a good vertex of $F$, but there are good vertices for which $\phi(v) = 0$, so the converse is not true.

\begin{definition} We will partition the set $\Vert F$ of size $d$ into three different groups:
  \begin{itemize}
    \item $A := \{u \in \Vert F \;:\; u \text{ is good}, \phi(u)\ne 0\}$,
    \item $B := \{v \in \Vert F \;:\; v \text{ is good}, \phi(v)= 0\}$,
    \item $C := \{w \in \Vert F \;:\; w \text{ is not good}\}$.
\end{itemize}
\end{definition}

The following result clarifies the geometric meaning of a good vertex being in $A$ or $B$.

\begin{lemma}\label{lem:one_vertex_at_0}
  Let $v$ be a good vertex of the facet $F$ (i.e., $v \in A \cup B$).  Then $\opp(F,w) \ne \opp(F,v)$ for every vertex $w$ of $F$ other than $v$ if and only if $\phi(v)\ne 0$ (i.e., $v \in A$).
\end{lemma}

\begin{proof}
  The `only if' part is proven in \cite[Lemma~25]{AJP}. For the converse assume that $\phi(v)\ne 0$ and $x := \opp(F,w) = \opp(F,v)$ for some vertex $w\not=v$ of $F$. By the second condition, $x$ and $v$ lie in the
  neighboring facet $\neigh(F,w)$. Hence, $x+v$ cannot be a vertex of $P$, a contradiction to $\phi(v)\ne 0$.
\end{proof}

\begin{example}\label{example:hexagon-phi}
  In \autoref{hexagon-fig-phi} the behaviour of the $\phi$-function is illustrated. The vertices $e_1$ and $e_2$ are both \emph{good} and belong to the partition set $A$. We have $\phi(e_1)=e_2$ and $\phi(e_2)=e_1$.
\end{example}

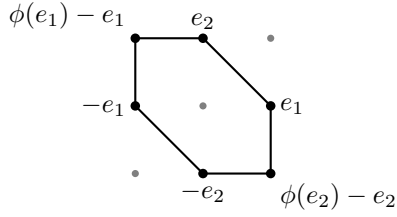
\begin{figure}[tb]
  \centering
		\begin{tikzpicture}[scale=0.9]
		  \tikzstyle{edge} = [draw,thick,-,black]
		
		  \foreach \x in {-1,0,1}
		    \foreach \y in {-1,0,1}
		       \fill[gray] (\x,\y) circle (1.5pt); ;
		
		  \coordinate (v0) at (0,0);
		  \coordinate (e1) at (1,0);
		  \coordinate (e2) at (0,1);
		  \coordinate (-e1+e2) at (-1,1);
		  \coordinate (-e2+e1) at (1,-1);
		  \coordinate (-e1) at (-1,0);
		  \coordinate (-e2) at (0,-1);
		
		  \draw[edge] (e1) -- (e2) -- (-e1+e2) -- (-e1) -- (-e2) -- (-e2+e1) -- (e1);
		
		  \foreach \point in {e1,e2,-e1+e2,-e2+e1,-e1,-e2}
		    \fill[black] (\point) circle (2pt);

      \node[right] at (e1){$e_1$};
      \node[above] at (e2){$e_2$};
      \node[above left] at (-e1+e2){$\phi(e_1)-e_1$};
      \node[left] at (-e1){$-e_1$};
      \node[below] at (-e2){$-e_2$};
      \node[below right] at (-e2+e1){$\phi(e_2)-e_2$};
		
		\end{tikzpicture}
  \caption{\label{hexagon-fig-phi} The hexagon $P_6$ with labeled vertices.}
\end{figure}

Let us deduce the following observations on the sizes of the sets $A,B,C$.

\begin{lemma}\label{ABC-card}
  \[
    d-2k \le |A|, \quad |B|+|C| \le 2k, \quad |C| \le k.
  \]
\end{lemma}

\begin{proof}
  By \autoref{lem:nill_5.5} every vertex on level $0$ must be opposite to some vertex in $F$, and clearly two different vertices on level $0$ cannot be opposite to the same vertex in $F$, as opposite vertices are unique. Therefore, the bound $\eta_0 \ge d-k$, see \autoref{card}, implies that at least $d-k$ vertices have pairwise distinct opposite vertices on level $0$. In particular, at least $d-k$ vertices of $F$ must be good, i.e., $|C| \le k$. Moreover, there can be at most $k$ of these $d-k$ good vertices, say $v$, that have the property that there exists another vertex $\bar v \in \Vert F$ with
  $\opp(F,\bar v)=\opp(F,v)$. Therefore, \autoref{lem:one_vertex_at_0} implies $|A| \ge d-2k$, and thus $|B|+|C| \le 2k$.
\end{proof}

\begin{example}
  The reader should be aware that we do not claim that $|B|\le k$. In fact, $|B| = 2k$ is possible as the following example shows. For this, consider the convex hull of the following $10$ points in dimension~$d=4$:
  \begin{gather*}
    e_1,\ e_2,\ e_3,\ e_4\\
    -e_1-e_2+e_3+e_4,\ -e_3-e_4+e_1+e_2\\
    -e_1,\ -e_2,\ -e_3,\ -e_4\,.
  \end{gather*}
  This polytope is indeed a smooth Fano polytope with $10$ vertices, so $k=2$, as $3\cdot d = 12$. The vertices $e_1,e_2,e_3,e_4$ form a special facet $F$. Here, all the vertices of $F$ are in the partition set $B$, e.g., $e_1$ and $e_2$ share the same opposite vertex $-e_1-e_2+e_3+e_4$. Therefore $|B| = 4 = 2k$. This example can be generalized to higher dimensions, where all the vertices of a special facet $F$ are good, and yet $k$ pairs of vertices share an opposite vertex.
\end{example}

Let us summarize what we know about the opposite vertices (keep in mind that $\Vert F$ forms a lattice basis).

\begin{proposition}\label{characterization:opp}
Let $z \in \Vert F$.
  \begin{enumerate}
    \item If $z \in A$, then $\opp(F,z) = -z + \phi(z)$.

    \item If $z\in B$, then
        \[
          \opp(F,z) = -z + \sum_{u\in A}a_u u + \sum_{\mathclap{v\in B\backslash\{z\}}}b_v v + \sum_{w\in C}c_w w
        \]
      with $a_u,c_w \in  \ZZ_{\ge 0}$ and $b_v \in \ZZ_{\ge -1}$ for all $u\in A$, $v\in B\backslash\{z\}$ and $w\in C$. Further it holds
        \begin{align*}
          \sum_{u\in A}a_u + \sum_{\mathclap{v\in B\backslash\{z\}}}b_v + \sum_{w\in C}c_w &= 1, &
          \sum_{u\in A} a_u &\le k+1,&
          \sum_{\mathclap{v\in B\backslash\{z\}}}b_v &\ge -k\,.
        \end{align*}

    \item If $z\in C$, then
        \[
          \opp(F,z) = -z + \sum_{u\in A}a_u u + \sum_{v\in B}b_v v + \sum_{\mathclap{w\in C\backslash\{z\}}}c_w w
        \]
      with $a_u,b_v,c_w \in \ZZ$ for all $u\in A$, $v\in B$ and $w\in C\backslash\{z\}$. Further it holds
        \[
          \sum_{u\in A}a_u + \sum_{v\in B}b_v + \sum_{\mathclap{w\in C\backslash\{z\}}}c_w \le 0\,.
        \]
  \end{enumerate}
\end{proposition}

\begin{proof}
  \begin{enumerate}
    \item Follows by definition.

    \item The first statement is a consequence of \autoref{lem:nill_5.5} and \autoref{lem:one_vertex_at_0}. The second equation follows since $\opp(F,v)$ is on level $0$. Finally, as in the proof of \autoref{ABC-card}, we note that for $z \in B$, there can be at most $k$ other vertices $\bar z \in \Vert F$
        (necessarily, $\bar z \in B$) with $\opp(F,\bar z) = \opp(F,z)$. By \autoref{lem:nill_5.5} this implies that there are at most $k$ other vertices $\bar z\in B$, $\bar z \not= z$, with $b_{\bar z} = \langle u_{F,\bar z}, \opp(F,z) \rangle <0$ (equivalently, equal to $-1$). We deduce that $\sum_{v\in B\backslash\{z\}}b_v \ge -k$. Since $\opp(F,z) \in V(F,0)$, this implies also that $\sum_{v\in A} a_u \le k+1$.

    \item This is as before a direct consequence of \autoref{minus} together with the fact that $z$ is not a good vertex and so the opposite vertex must lie on a level below $0$.
  \end{enumerate}
\end{proof}

\subsection{Vertices on negative levels}
Let us recall some of \Obro's observations.

\begin{lemma}[{\Obro~\cite[Lem.~1 and~2]{Obro08}}]\label{lem:1+2}
  Let $G$ be a facet of $P$, $z$ a vertex of $G$, and $G':=\neigh(G,z)$.
  Then $u_{G'} = u_G + (\langle u_{G'},z \rangle -1)u_{G,z}$. Moreover, for $x\in P$ we have
  $\langle u_{G}, x \rangle -1 \le \langle u_{G,z}, x \rangle$. In case of
  equality we have $x = \opp(G,z)$.
\end{lemma}

\begin{corollary}\label{cor-star}
  If $z \in A \cup B$, then $u_{F'} = u_F - u_{F,z}$ for $F' := \neigh(F,z)$.
\end{corollary}

\begin{proof}
This follows from plugging in $x:=\opp(F,z)$ into \autoref{lem:1+2}.	\end{proof}

The following observation is contained in \Obro's PhD thesis \cite[Lemma 1.9(5)]{OebroPhD}.

\begin{lemma}\label{lem:opposite_is_shortest}
  Let $x$ be a vertex of $P$ with $x\ne \opp(F,z)$ for some vertex $z \in \Vert F$. Then the inequality $\langle u_{F,z}, x \rangle < 0$ implies $\langle u_F , x \rangle < \langle u_F, \opp(F,z) \rangle$. In other words, if $z$ has a negative contribution to $x$ then the vertex $x$ must lie on a level strictly lower than the opposite vertex of $z$.
\end{lemma}

As we know that for all $w\in C$ the vertex $\opp(F,w)$ lies strictly below level $0$, we get the following consequence.

\begin{corollary}\label{cor:bad+level-1_implies_opposite}
  There is at most one vertex $x\in V(F,-1)$ with $\langle u_{F,w}, x \rangle <0$ for some $w\in C$. If such a vertex $x$ exists, then $x = \opp(F, w)$.
\end{corollary}

Let us make some auxiliary considerations regarding vertices on negative levels.

\begin{lemma}\label{lem:no_two_verts_with_very_low_coordinate}
  Let $z\in A\cup B$. For every $\bar z \in \Vert F$ there is at most one vertex $x \in V(F,-1)$ with $\langle u_{F,z}, x \rangle = -1$ and $\langle u_{F,\bar z},x \rangle < \langle u_{F,\bar z}, \opp(F,z) \rangle$.
  %In other words, the coefficients of the vertices in $V(F,-1)$ are bounded from below by the coefficients of the opposite vertices of good vertices, with at most one exception.
\end{lemma}

\begin{proof}
  Let us denote $y=\opp(F,z)$. We have $\pro{u_F}{y} =0$, since $z$ is good. Suppose there are two vertices $x,\bar x \in V(F,-1)$ with $\langle u_{F,z}, x \rangle = \langle u_{F,z}, \bar x \rangle = -1$ and $\langle u_{F,\bar z},x \rangle < \langle u_{F,\bar z}, y \rangle$ as well as $\langle u_{F,\bar z},\bar x \rangle < \langle u_{F,\bar z}, y \rangle$ for some $\bar z\in \Vert F$.
  We rewrite $x$, $\bar x$ and $y$ as:
    \begin{align*}
      x &= \sum_{\mathclap{v\in \Vert F}} \overbrace{\langle u_{F,v}, x \rangle}^{\lambda_v:=} v\,, &
      \bar x &= \sum_{\mathclap{v\in \Vert F}} \overbrace{\langle u_{F,v}, \bar x \rangle}^{\mu_v:=} v\,, &
      y &= \sum_{\mathclap{v\in \Vert F}} \overbrace{\langle u_{F,v}, y \rangle}^{\alpha_v:=} v.
    \end{align*}
 Considering the neighboring facet $F' := \neigh(F,z)$ and its outer facet normal $u_{F'}$, \autoref{cor-star} implies $u_{F'} = u_F  - u_{F,z}$, so $x$ and $\bar x$ both lie on level $0$ with respect to $F'$. We rewrite $x$ and $\bar x$ in terms of the new
  lattice basis $\Vert F' = \Vert F \setminus \{ z \} \cup \{ y \}$:
  \begin{align*}
    x &= y +\sum_{v\ne z} \underbrace{(\lambda_v - \alpha_v)}_{=\langle u_{F',v}, x\rangle} v\,,
    & \bar x &= y+\sum_{v \ne z} \underbrace{(\mu_v - \alpha_v)}_{=\langle u_{F',v}, \bar x\rangle} v,
  \end{align*}
where we used $\lambda_z = \mu_z = \alpha_z=-1$. Our assumptions on $\bar z$ gives us that $\lambda_{\bar z} < \alpha_{\bar z}$ and $\mu_{\bar z} < \alpha_{\bar z}$. So \autoref{lem:nill_5.5} yields
$x = \opp(F',\bar z) = \bar x$.
\end{proof}

\begin{lemma}\label{lem:at_least_one_vertex_with_low_coordinate}
  Let $x,\bar x \in \Vert P$ with $\langle u_F,x \rangle < \langle u_F,\bar x \rangle$. Then there must exist at least one vertex $z\in \Vert F$ with $\langle u_{F,z},x \rangle < \langle u_{F,z}, \bar x \rangle$.
\end{lemma}

\begin{proof}
  Suppose this is not true, so $\langle u_{F,z},x \rangle \ge \langle u_{F,z}, \bar x \rangle$ for every $z\in \Vert F$. Then we get a contradiction from
    \[
      \langle u_F,x \rangle = \sum_{z\in \Vert F}\langle u_{F,z},x \rangle \ge \sum_{z\in\Vert F} \langle u_{F,z}, \bar x \rangle =  \langle u_F,\bar x \rangle.
    \]
\end{proof}

\begin{lemma}[{\Obro~\cite[Lem.~5]{Obro08}}]\label{lem:phi:notwonegative}
  Let $v$ and $w\ne v$ be good vertices of the facet $F$ with $\opp(F,v) \ne \opp(F,w)$. Then there exists no vertex $x\in V(F,-1)$ such that $\langle u_{F,v}, x\rangle = \langle u_{F,w}, x\rangle = -1$.
\end{lemma}

Now we can give a rough classification of vertices in $V(F,-1)$.

\begin{proposition}\label{prop:number_of_-e_i}
  Any vertex $x \in V(F,-1)$ is one of the following three types:
    \begin{enumerate}[label=(\roman*), ref=\roman*, leftmargin=*]
      \item \label{prop:number_of_-e_i:i}
        $x = \opp(F,z)$ for some $z \in C$ with $\pro{u_{F,z}}{x} < 0$.

      \item \label{prop:number_of_-e_i:ii}
        $\pro{u_{F,z}}{x} =-1$ for some $z \in B$,  $\pro{u_{F,v}}{x} \ge-1$ for all $v \in B$, and $\pro{u_{F,w}}{x} \ge 0$ for all $w \in A \cup C$.

      \item \label{prop:number_of_-e_i:iii}
      $x=-z$ for some $z \in A$.
    \end{enumerate}
  There are at most $|C|$ vertices of type \eqref{prop:number_of_-e_i:i}, at most $(k+1)|B|$ vertices of type \eqref{prop:number_of_-e_i:ii}, and at least $\eta_{-1} -|C| - |B|(k+1)$
  vertices of type \eqref{prop:number_of_-e_i:iii}.
\end{proposition}

\begin{proof}
  We consider the following situation: $x \in V(F,-1)$ and $z\in \Vert F$ with $\langle u_{F,z}, x \rangle <0$ (clearly, such a $z$ must exist). The proof proceeds in three steps.

\smallskip

  First, suppose $z \in C$.  \autoref{cor:bad+level-1_implies_opposite} implies that $x$ is the only vertex on level $-1$ which satisfies
  $\langle u_{F,z},x \rangle<0$ (namely, the opposite vertex to $z$). Hence, $x$ is of type \eqref{prop:number_of_-e_i:i}.

\smallskip

  Secondly, let us look at the remaining vertices $x$ on level $-1$ (of which there are at least $\eta_{-1} - |C|$ many), i.e., $\pro{u_{F,\bar{z}}}{x} \ge 0$ for any $\bar{z} \in C$. Suppose
  we have $z\in B$ with $\langle u_{F,z}, x \rangle <0$. \autoref{lem:1+2} implies $\langle u_{F,z}, x \rangle=-1$. We define $y := \opp(F, z) \in V(F,0)$. We claim that
  $\langle u_{F,\bar z},x \rangle\ge0$ for all $\bar z \in \Vert F$ which satisfy $\opp(F,\bar z) \ne y$. Clearly, this is automatically true for $\bar z \in C$ by our choice of $x$. Otherwise, $\bar z$ is good,
  and the claim follows from \autoref{lem:phi:notwonegative}. Hence, we get (use Lemmas~\ref{lem:nill_5.5}, \ref{minus} and \ref{lem:1+2} again):
    \begin{equation}\label{eq:bounds}
      \begin{aligned}
        \bar z\in \Vert F \text{ with } \opp(F, \bar z) \ne y &\implies \langle u_{F,\bar z}, y\rangle \ge 0 \text{ and } \langle u_{F,\bar z}, x\rangle \ge 0 \\
        \bar z\in \Vert F \text{ with } \opp(F, \bar z)  =  y &\implies \langle u_{F,\bar z}, y\rangle   = -1 \text{ and } \langle u_{F,\bar z}, x\rangle \ge -1
      \end{aligned}
    \end{equation}
  Here we see that $\langle u_{F,\bar z}, x\rangle = -1$ is only possible if $\bar z \in B$ with $\opp(F,\bar z) = y$ (see also \autoref{characterization:opp}\eqref{prop:number_of_-e_i:ii}).
  From these observations we deduce that $x$ is of type \eqref{prop:number_of_-e_i:ii}. Now, \autoref{lem:at_least_one_vertex_with_low_coordinate} implies that there must exist at least one vertex $\bar z \in \Vert F$ with $\langle u_{F,\bar z}, x \rangle < \langle u_{F,\bar z}, y \rangle$ because $x$ lies on level $-1$ and $y$ is on level $0$.
  From \eqref{eq:bounds} we observe that this implies
  $\langle u_{F,\bar z},y \rangle > 0$. By \autoref{characterization:opp}\eqref{prop:number_of_-e_i:ii}, $y$ can
  have at most $(k+1)$ many negative coordinates and lies on level $0$
  so we see that $\langle u_{F,\bar z},y \rangle > 0$ is only possible for at most $(k+1)$
  choices for $\bar z \in \Vert F$. But \autoref{lem:no_two_verts_with_very_low_coordinate} tells us that $x$
  is uniquely determined by $x \in V(F,-1)$, $\pro{u_{F,z}}{x} = -1$ and
  $\langle  u_{F,\bar z}, x \rangle < \langle u_{F,\bar z}, y \rangle$ for given $\bar z \in \Vert{P}$. Therefore, for given vertex $z \in B$, there are at most $(k+1)$ vertices $x \in V(F,-1)$ that are not opposite to a vertex in $C$ and satisfy $\langle u_{F,z}, x \rangle <0$.

\smallskip

  Finally, consider the remaining vertices $x$ on level $-1$, of which there are at least $\eta_{-1} -|C| - |B|(k+1)$ many, i.e.,
  $\langle u_{F,\bar z}, x \rangle \ge 0$ for all $\bar z \in B\cup C$. In this case, we must have $z\in A$ with $\langle u_{F,z}, x \rangle < 0$ for such an $x$.
  Again, \autoref{lem:1+2} shows that this is equivalent to $\langle u_{F,z}, x \rangle=-1$. Now, \autoref{lem:phi:notwonegative} implies
  $\langle u_{F,\bar z},x \rangle\ge0$ for all $\bar z\in A\setminus\{ z \}$.
  Since $\sum_{v\in \Vert F} \langle u_{F,v},x \rangle v = x$ with $\sum_{v\in \Vert F} \langle u_{F,v},x \rangle = -1$ and there is only one negative summand which is already $-1$, we must have $x=-z$, so $x$ is of type \eqref{prop:number_of_-e_i:iii}.
\end{proof}

\subsection{Construction of the splitting}

We will now restrict our set $A$.

\begin{definition}
  Let
    \[
      A' := \{v \in A \;:\; -v \in \Vert{P}\}.
    \]
\end{definition}

\begin{lemma}[{Assarf,~Joswig,~Paffenholz \cite[Lemma~29]{AJP}}]\label{lem:phiphi}
  If $v \in A'$, in particular $\phi(v)\ne 0$, then $\phi(\phi(v)) \in \{ 0,v \}$.
\end{lemma}

We want to restrict $A'$ further to a subset of elements $v$ that have the desirable property $\phi(\phi(v))=v$.
%Hence, we want to limit the number of vertices where $\phi(\phi(v)) = 0$.
For this, let us first recall the following result.

\begin{lemma}[{Casagrande~\cite[Lemma~3.3]{CasagrandeBirational}}]\label{casa-lemma}
  If $z,x,x',y,y' \in \Vert{P}$ with $x \not= x'$ such that $x+z=y$ and $x'+z=y'$, then $y'=-x$.
\end{lemma}

\begin{corollary}\label{lem:bound_on_delta}
  Let $z\in \Vert F$. Then
 \[
   |\SetOf{v \in A'}{\phi(v) = z}| \le 1.
 \]
\end{corollary}

\begin{proof}
  Assume there exist $v,v' \in A'$, $v \not= v'$ with $\phi(v)=z=\phi(v')$. We define
  \begin{align*}
    x  &:= -v, &
    x' &:= -v', &
    y  &:= \opp(F,v)=z-v, &
    y' &:= \opp(F,v')=\phi(v')-v'
  \end{align*}
  Then \autoref{casa-lemma} implies $\phi(v')-v' = y' = -x = v$, a contradiction since $\phi(v')-v'$ is on level $0$, while $v$ is on level $1$.
\end{proof}

  Next, let us write the vertex sum as a linear combination of vertices of $F$:
  \[
    s_P = \sum_{z \in \Vert{F}} \gamma_z \, z.
  \]
Note that $\gamma_z \ge 0$ for any $z \in \Vert{F}$. Now, we can restrict the set of good vertices further.

\begin{definition}
  Let
    \[
      \bar A := \{v \in A' \;:\; \phi(v) \in A',\; \gamma_v = 0 = \gamma_{\phi(v)}\}.
    \]
  Note that \autoref{lem:phiphi} implies $\phi(\phi(v)) = v$ for $v\in \bar A$. In particular, $v \in \bar A$ implies $\phi(v) \in \bar A$.
\end{definition}

The reason for including the condition on the vertex sum in the definition of $\bar A$ is that changing the special facet will be useful in the upcoming proofs, and this condition will ensure that
the new facet will still be special.

\medskip

Note that $\bar A$ splits into disjoint pairs, each pair, say $v$, $\phi(v)$, giving rise to the following configuration of vertices of $P$ (see also \autoref{example:hexagon-phi} and \autoref{hexagon-fig-phi}):
  \begin{gather*}
    v \,,\ \phi(v)\\
    v - \phi(v) \,,\ \phi(v)-v\\
    -\phi(v) \,,\ -v
  \end{gather*}

We can now estimate the size of the set $\bar A$.

\begin{proposition}\label{prop:chuck_norris_vertices}
  \[
    |\bar A| \ge 2 |A'| - d -2k.
  \]
\end{proposition}

\begin{proof}
  We have by \autoref{lem:bound_on_delta}
    \[
      |\{v \in A' \;:\; \phi(v) \in B \cup C\} \,\cup\, \{v \in A' \;:\; \phi(v) \in A\setminus A'\}| \le |B|+|C| + |A| - |A'|.
    \]
  This shows that
    \[
      \{v \in A' \;:\; \phi(v) \in A'\} \ge |A'| - (|B|+|C| + |A| - |A'|) = 2 |A'| - d.
    \]
  Finally, from \autoref{sum-level} we see that there are at most $k$ vertices in $F$ with nonzero $\gamma$-coordinate. Hence,
  we have to potentially remove $k$ pairs, so $2k$ vertices from the previous estimate.
\end{proof}

In order to finally prove that we indeed have a splitting into a sufficiently large subset of these hexagons it is necessary to bound the number of its nonzero coordinates in $\bar A$ for each vertex on a negative level.

\begin{lemma}\label{prop:chuck_norris_times_2_vertices}
  For $x\in V(F,\le -1)$ it holds
    \[
      |\SetOf{z\in \bar A}{\langle u_{F,z}, x \rangle \ne 0}| \le 2k+2.
    \]
\end{lemma}

\begin{proof}
  Let $x$ be a vertex in $V(F,\le -1)$. Suppose there exists a vertex $v_0\in \bar A$ with $\langle u_{F,v_0}, x \rangle < 0$, then we consider the
  neighboring facet $F_1 := \neigh(F,v_0)$. Note that by our assumption on $\bar A$, the facet $F_1$ is still a special facet, as $F_1$ has vertices $\Vert F \setminus \{v_0\} \cup \{\phi(v_0)-v_0\}$. Hence, $u_{F_1} = u_F - u_{F,v_0}$.
  In particular, we get $\pro{u_{F_1}}{x} > \pro{u_F}{x}$.

  Now, suppose there exists another vertex $v_1 \in \bar A$ with $v_1 \not\in \{v_0, \phi(v_0)\}$ such that $\langle u_{F,v_1}, x \rangle < 0$.  Note that by the choice of $v_1$ and by \autoref{lem:1+2} and the fact that $u_{F_1,v_1}=u_{F,v_1}$, the opposite vertex of $v_1$ with respect to $F_1$ is the same as for $F$, that is $\opp(F_1,v_1) = \opp(F,v_1) = \phi(v_1)-v_1$.
  Consider $F_2 := \neigh(F_1,v_1)$. Again, by \autoref{lem:nill_5.5},
  the facet $F_2$ has vertices
   \[
    \Vert F \setminus \{v_0, v_1\} \cup \{\phi(v_0)-v_0\} \cup \{\phi(v_1)-v_1\}
   \]
  and therefore is also a special facet. Hence, $u_{F_2} = u_F - u_{F,v_0} - u_{F,v_1}$.
  In particular, $\pro{u_{F_2}}{x} > \pro{u_{F_1}}{x}$.

  Continuing with this argumentation we might end up with a facet $F_\ell$ where $\pro{u_{F_\ell}}{x}=0$ and vertices
    \[
      \Vert F \setminus \{v_0, \ldots, v_\ell\} \cup \{\phi(v_0)-v_0\} \cup \cdots \cup \{\phi(v_{\ell-1})-v_{\ell-1}\}\,.
    \]
  At this moment there cannot exist another vertex $v_\ell\in \bar A$ with $\langle u_{F, v_\ell}, x \rangle <0$, since \autoref{lem:nill_5.5} shows that $x = \opp(F_\ell,v_\ell)=\phi(v_\ell)-v_\ell \in V(F,0)$ (note that  $u_{F_\ell,v_\ell}=u_{F,v_\ell}$ and due to our choices the opposite vertices of $\Vert F_\ell \cap \bar A$ do not change). A contradiction.

  So the number of vertices $v$ with $\langle u_{F,v}, x \rangle <0$ is at most $2\cdot |\langle u_F, x \rangle|$, where the factor $2$ comes from the fact that $\langle u_{F,\phi(v_i)}, x \rangle$ could be zero or not.

  One can repeat the same argument with $\langle u_{F,v}, x \rangle >0$. In this case, since all the facets were special, we can use that $-k-1$ is the lowest possible level (\autoref{card}) to deduce that there are at most $2\cdot(k+1 - |\langle u_F, x \rangle|)$ vertices $v \in \bar A$ with $\langle u_{F,v}, x \rangle >0$. Putting this together we get the desired statement.
\end{proof}

\subsection{The proof of \autoref{thm:main}}
\label{proof-of-theo}
Let us define
  \[
    W := \bigcup_{v \in \bar A} \{v,\ \phi(v),\ \phi(v)-v,\ v-\phi(v),\ -v,\ -\phi(v)\}\,.
  \]
Note that the convex hull of these vertices splits into hexagons as desired.
Our goal is now to bound the number of vertices $v \in \bar A$ for which there exists some vertex $x \in \Vert{P}\setminus W$ with $\pro{u_{F,v}}{x} \not=0$. Then
after removing these undesired vertices $v$ (and their hexagons) from $W$, the remaining vertices of $W$ live in a subspace that is transversal to all
other vertices of $P$.

\medskip

First let us consider the case that $x \in V(F,0)\backslash W$ with $\pro{u_{F,v}}{x} \not=0$ for some $v \in \bar A$. In particular, $x$ is not opposite to any vertex in $\bar A$ (since $x \not\in W$). By \autoref{lem:nill_5.5}, this implies $\pro{u_{F,v}}{x} > 0$, and $x = \opp(F,v')$ for some $v' \in \Vert{F}$ with $v' \in (A \setminus \bar A) \cup B$. We distinguish
two cases. If $v' \in A\setminus \bar A$, then \autoref{characterization:opp} implies $\phi(v')=v$ and $x=v-v'$. (Since in this case each choice of $v'$ uniquely determines $x$ and thus $v$, there are at most $|A \setminus \bar A|$ many possibilities for $v$.)
We can say more for $v' \in B$. Let us partition $B$ into disjoint subsets $B_1, \ldots, B_t$, where two vertices of $B$ are in the same subset $B_i$ (for $i \in \{1, \ldots, t\}$) if and only if their opposite vertices are the same. Let us assume that $v' \in B_j$ for some $j \in \{1, \ldots, t\}$. By \autoref{lem:nill_5.5} and \autoref{characterization:opp}, the vertex $x$ has negative coordinates only in $B_j$, so $x$ has at most $|B_j|$ many positive coordinates in $\bar A$, since $x$ is on level $0$. 
So, in this second case there are at most $t$ possibilities for $x$, and at most $|B_1| + \cdots +
|B_t| = |B|$ many possibilities for $v$.
Hence, the number of vertices $v \in \bar A$ for which there exists some vertex $x \in V(F,0)\setminus W$ with $\pro{u_{F,v}}{x} \not=0$ is bounded by
\begin{equation}\label{final-eq1}
  |A| - |\bar A| + |B|.
\end{equation}

Next we notice that there are at most $2k$ vertices in $V(F,\le -2)$ by \autoref{card}. For each such vertex $x$, \autoref{prop:chuck_norris_times_2_vertices} implies that $\langle u_{F,v}, x \rangle \ne 0$ can be satisfied for at most $2k+2$ many choices of $v\in \bar A$. Therefore, the number of vertices $v \in \bar A$ for which there exists some vertex $x \in V(F,\le -2)\setminus W$ with $\pro{u_{F,v}}{x} \not=0$ is bounded by
\begin{equation}\label{final-eq2}
  (2k+2)2k.
\end{equation}

This leaves the case where $x\in V(F,-1)$ with $-x \not\in \bar A$ and $\langle u_{F,v},x \rangle \not=0$ for some $v \in \bar A$. By \autoref{prop:number_of_-e_i} the vertex $x$ must be of type \eqref{prop:number_of_-e_i:i} or \eqref{prop:number_of_-e_i:ii}.
\autoref{prop:chuck_norris_times_2_vertices} shows that the number of vertices $v \in \bar A$ for which there exists some vertex $x \in V(F,-1)\setminus W$ of type \eqref{prop:number_of_-e_i:i} with $\pro{u_{F,v}}{x} \not=0$ is bounded by
\begin{equation}\label{final-eq3}
  (2k+2)|C|.
\end{equation}

Let us define two sets. $\tilde{V}$ is the set of vertices in $V(F,-1)$ of type \eqref{prop:number_of_-e_i:ii}, and $\tilde{A}$ is the set of vertices $v \in \bar A$ where there exists some $x \in \tilde{V}$ with $\pro{u_{F,v}}{x} \not= 0$
(necessarily, $> 0$). Note that by the description of a vertex $x \in \tilde{V}$ in \autoref{prop:number_of_-e_i} it holds that, since $x \in V(F,-1)$, if $j$ many coordinates of $x$ in $\bar A$ are positive, then at least $j+1$ many coordinates of $x$ in $B$ must be equal to $-1$. Hence,
\begin{align*}
  |\tilde{A}| + |\tilde{V}| &\le \sum_{x \in \tilde{V}} (|\{v \in \tilde{A} \,:\, \pro{u_{F,v}}{x} > 0\}| + 1) \\
  &\le \sum_{x \in \tilde{V}} |\{v \in B \,:\, \pro{u_{F,v}}{x} =-1\}| \\
  &= \sum_{v \in B} |\{x \in \tilde{V} \,:\, \pro{u_{F,v}}{x} =-1\}| \\
  &\le  (k+1) |B|,
\end{align*}
where we used \autoref{prop:number_of_-e_i} for the last inequality. This shows that the number of vertices $v \in \bar A$ for which there exists some vertex $x \in V(F,-1)\setminus W$ of type \eqref{prop:number_of_-e_i:ii} with $\pro{u_{F,v}}{x} \not=0$ is bounded by
\begin{equation}\label{final-eq4}
  (k+1) |B| - \tilde{V}.
\end{equation}

Summing up Equations~\eqref{final-eq1},\eqref{final-eq2},\eqref{final-eq3}, and \eqref{final-eq4} we get that the number of vertices $v \in \bar A$ for which there exists some vertex $x \in \Vert{P}\setminus W$ with $\pro{u_{F,v}}{x} \not=0$ is bounded by
  \[
    |A| - |\bar A| + |B| + (2k+2)2k + (2k+2)|C| + (k+1) |B| - |\tilde{V}|.
  \]
Using $|A|+|B|+|C|=d$ and $|\tilde{V}| \ge \eta_{-1} - |C| - |A'|$ (see \autoref{prop:number_of_-e_i}), this expression can be bounded from above by
  \[
    (2k+2)2k+(2k+2)|C|+(k+1)|B| + |A'| - |\bar A| - \eta_{-1} + d.
  \]

\medskip

Now, we want to remove all these vertices $v$ together with $\phi(v)$. So, we deduce that the number of vertices $v \in \bar A$ for which there exists \emph{no} vertex $x \in \Vert{P}\setminus W$ with $\pro{u_{F,v}}{x} \not=0$ or $\pro{u_{F,\phi(v)}}{x} \not=0$ is \emph{at least}
  \[
    \begin{multlined}
      |\bar A| - 2 \Big((2k+2)2k+(2k+2)|C|+(k+1)|B| + |A'| - |\bar A| - \eta_{-1} + d\Big)\\
      = 3 |\bar A| - 8(k+1)k-4(k+1)|C|-2(k+1)|B| - 2 |A'| + 2 \eta_{-1} - 2 d.
    \end{multlined}
  \]
From $|\bar A| \ge 2 |A'| - d -2k$ (\autoref{prop:chuck_norris_vertices}) we get that this expression is at least
  \[
    4 |A'| -6k- 8(k+1)k-4(k+1)|C|-2(k+1)|B| + 2 \eta_{-1} - 5 d.
  \]
 Using $|A'| \ge \eta_{-1} -|C| - |B|(k+1)$ (\autoref{prop:number_of_-e_i})
we get that the previous expression is at least
  \[
    -6k- 8(k+1)k-4(k+2)|C|-3(k+1)|B| + 6 \eta_{-1} - 5 d.
  \]
Plugging in $\eta_{-1} \ge d-2k$ (\autoref{card}) we deduce that the previous expression is at least
  \[
    -18k- 8(k+1)k-4(k+2)|C|-3(k+1)|B| + d.
  \]
Finally, we use $|B|+|C| \le 2k$ and $|C| \le k$ (\autoref{ABC-card}) to show that the previous expression is at least
  \[
    d-15k^2-37k.
  \]
Summing up, we have proved that there are at least $d-15k^2-37k$ vertices $v \in \bar A$ such that all vertices $x \in \Vert{P}\setminus W$ satisfy $\pro{u_{F,v}}{x} =0$ and $\pro{u_{F,\phi(v)}}{x} =0$. Hence, if $d \ge 15k^2-37k + 2$, then $P$ splits such that there are at least
  \[
    \left\lfloor \frac{d-15k^2-37k}{2}\right\rfloor \ge 1
  \]
hexagons $P_6$ as splitting factors. This finishes the proof.

\hfill \qed

\begin{remark}\label{remark:bottle-neck}
The reader may wonder why our proof does not achieve a linear bound. In fact, the quadratic order comes from \eqref{final-eq2}, \eqref{final-eq3}, and \eqref{final-eq4}. These inequalities rely on several worst-case bounds such as \autoref{prop:chuck_norris_times_2_vertices}. The main bottleneck at this crucial point in the proof is the
separate application of these estimates when the worst cases may actually exclude each other. Hence, in order to get an overall linear bound it would be desirable
to find a direct and unified approach to these estimates.
\end{remark}

\section{Extremal smooth Fano polytopes}
\label{sec:extreme}
In this section, we explain why \autoref{conjecture1'} and \autoref{conjecture2} would be sharp. These conjectures were found using \polymake~\cite{DMV:polymake}. We analyzed all smooth Fano polytopes up to dimension $9$ (and checked the validity) using the database extension \texttt{poly\_db}~\cite{poly_db}, where all the polytopes found in \cite{smoothreflexive} are part of the database.

\begin{definition} Let $d = 3\cdot d'$ for some positive integer $d'$. Let $e_1, \ldots, e_d$ be the standard basis of $\RR^d$. We define $v$ to be the following vector in $\RR^d$
    \[
      v:= \sum_{i=1}^{d'}e_{d'+2i} - \sum_{i=1}^{d'}e_i\,.
    \]
  It consists of $-1$ in the first $d'$ coordinates and then the coordinates alternate between $0$ and $1$. We define $S$ to be the simplex $S := \conv\{ v, e_1, e_2, \ldots, e_{d'} \} \subseteq \RR^d$. Let $P_6 \subset \RR^2$ be the hexagon from \autoref{hexagon-fig}. Finally, we define $B_{d'}$ as the combined convex hull
    \[
      B_{d'} := \conv \left( S \;\cup\; \left(\{ 0 \}^{d'}\times P_6^{\oplus d'}\right)\right)\,.
    \]
\end{definition}

\begin{remark} Let us recall the following construction (generalizing \autoref{free-sum-def}). Given two polytopes $Q,Q' \subset \RR^d$ whose affine hulls intersect precisely in a point (not necessarily a lattice point), then the convex hull of
$Q$ and $Q'$ is a {\em combinatorial free sum} (also called direct sum or linear join) of $Q,Q'$. This means that its combinatorial type is dual to the product
of the combinatorial dual polytope of $Q$ and the combinatorial dual polytope of $Q'$. In particular, the combinatorial type of a combinatorial free sum of $Q$ and $Q'$ only depends on the combinatorial types of $Q$ and $Q'$.
\label{free-sum-remark}\end{remark}

\begin{proposition}
  $B_{d'}$ is a smooth Fano polytope of dimension $3d'$ with $7 d'+1$ vertices that does not split into lower-dimensional smooth Fano polytopes.
\end{proposition}

\begin{proof}
  The polytope $Q := \{ 0 \}^{d'}\times P_6^{\oplus d'}$ is a $2d'$-dimensional smooth Fano polytope
  (considered as $P_6^{\oplus d'}$ in $\RR^{2d'}$) with $6 d'$ vertices. Since $S$
  is a $d'$-dimensional simplex containing $e_1, \ldots, e_{d'}$, the polytope $B_{d'}$ has clearly
  dimension $3d'$ and thus is full-dimensional. We observe that the barycenter of $S$ equals
  $x := \sum_{i=1}^{d'}\frac{1}{d'+1} e_{d'+2i}$ which is a point in the interior of $Q$.
  This is the only intersection point of $Q$ and $S$. Hence, in the notation of \autoref{free-sum-remark},
  $B_{d'}$ is a combinatorial free sum of $Q$ and $S$. In particular, $B_{d'}$ is simplicial and every facet $F$ of $B_{d'}$ is
  the convex hull of a facet in $S$ and a facet in $Q$.

  Our goal is now to show that each facet forms a lattice basis.
  As $P_6^{\oplus d'}$ is a $2d'$-dimensional smooth polytope we
  know that $F$ contains a lattice basis for the subspace $\{ 0 \}^{d'}\times \RR^{2d'}$. Let us distinguish whether
  $v\in F$ or not. Suppose $v\not\in F$ then the facet contains the points $e_1$, $e_2$, $\ldots$, $e_{d'}$, which is a
  lattice basis of the space $\RR^{d'}\times \{ 0 \}^{2d'}$. If on the other hand $v\in F$ then one of those cartesian
  basis vectors is missing, let this be $e_1$ without loss of generality. It is sufficient to show that we can write
  $e_1$ as an integer linear combination of vertices of $F$. Consider $e_1 = -v - \sum_{i=2}^{d'}e_i +
  \sum_{i=1}^{d'}e_{d'+2i}$. This integer linear combination consists either of vertices of $F$ or of points which can
  be expressed as an integer linear combination of vertices of $F$ since we already know that the vertices form a
  lattice basis for $\{ 0 \}^{d'}\times \RR^{2d'}$. This shows that $B_{d'}$ is a smooth Fano polytope.

  Finally, assume that $B_{d'}$ would decompose as a direct sum of two lower-dimensional
  smooth Fano polytopes $P_1$ and $P_2$. Say, $v$ is a vertex of $P_1$. Since $v$ is a
  linear combination of $e_1, \ldots, e_{d'}$ and $e_{d'+2i}$ for $i=1, \ldots, d'$, these vertices
  must also be vertices of $P_1$.  However, no hexagon $P_6$ splits, so all vertices of $P$ are in
  $P_1$, a contradiction.
\end{proof}

Regarding the sharpness of \autoref{conjecture1'}, we see that $B_1^{\oplus m}$ has dimension $d := 3 m$ with $8 m$ vertices. Hence, the associated
toric Fano manifold has Picard number $5m = \frac{5d}{3}$. This example and its toric variety was already considered in \cite{AJP}.

Regarding the sharpness of \autoref{conjecture2}, we observe that the toric Fano manifold associated to $B_{d'}$ has dimension $d := 3 d'$ and Picard number $\frac{4d+3}{3}$. We refer
to the book \cite{CLSbook} for the explanation why this defines a toric fiber bundle as stated in \autoref{conjecture2}.

\begin{acknowledgment}
  The first author thanks Stockholm University for hospitality and the Berlin Mathematical School for its support. The second author is partially supported by the Vetenskapsr\aa det grant NT:2014-3991. We are grateful to Cinzia Casagrande for providing us with \autoref{alg-geo-remark}. The authors would also like to thank the anonymous referees for several suggestions on  how to improve the presentation of the paper.
\end{acknowledgment}

\bibliographystyle{amsplain}
\bibliography{main}

\end{document}